\documentclass{article}
\usepackage{graphicx}
\usepackage{tabularx,multirow}
\usepackage{enumerate,enumitem}
\usepackage{amsmath,amsthm}
\usepackage{amsfonts,amssymb}
\usepackage[numbers,sort]{natbib}
\usepackage[affil-it]{authblk}

\DeclareGraphicsExtensions{.pdf,.png,.eps}
\graphicspath{{figs/}}

\DeclareMathOperator{\outdeg}{out-deg}
\DeclareMathOperator{\indeg}{in-deg}

\renewcommand{\P}{\ensuremath{\mathcal{P}}}

\newtheorem{theorem}{Theorem}
\newtheorem{corollary}[theorem]{Corollary}
\newtheorem{lemma}[theorem]{Lemma}

\begin{document}

\title{Splitting Planar Graphs of Girth $6$ into Two Linear Forests with Short Paths.}

\author{Maria Axenovich}
\author{Torsten Ueckerdt}
\author{Pascal Weiner}
\affil{Karlsruhe Institute of Technology, Germany}
\renewcommand\Authands{ and }

\maketitle

\begin{abstract}
 Recently, Borodin, Kostochka, and Yancey (On $1$-improper $2$-coloring of sparse graphs. \textit{Discrete Mathematics}, 313(22), 2013) showed that the vertices of each planar graph of girth at least $7$ can be $2$-colored so that each color class induces a subgraph of a matching. 
 We prove that any planar graph of girth at least $6$ admits a vertex coloring in $2$ colors such that
 each monochromatic component is a path of length at most $14$.
 Moreover, we show a list version of this result.
 On the other hand, for each positive integer $t\geq 3$, we construct a planar graph of girth $4$ such that in any coloring of vertices in $2$ colors there is a monochromatic path of length at least $t$. 
 It remains open whether each planar graph of girth $5$ admits a $2$-coloring with no long monochromatic paths.
\end{abstract}

 \section{Introduction}

In this paper, we consider the question of partitioning the vertex set of a planar graph into a small number of parts such that each part induces a graph whose connected components are short paths.
Equivalently, we consider vertex-colorings of planar graphs such that each monochromatic component is a  short path.
The length of a path is the number of its edges.
The Four Color Theorem~\cite{4CT1,4CT2} implies that four parts are sufficient to guarantee such a partition with paths of length $0$, i.e., on $1$ vertex each.
A result of Poh~\cite{Poh90} shows that any planar graph can be vertex-colored with $3$ colors such that each monochromatic component is a path. 
Chappel constructed an example of a planar graph whose largest induced subgraph with path components has   at most $4/9n$ vertices, where $n$ is the total number of vertices, see~\cite{Pel} for the construction.
This shows in particular, that the result of Poh is tight.
However, one can not restrict the lengths of monochromatic paths in $3$-colorings of planar graph as was shown by a specific triangulation construction of Chartrand, Geller and Hedetniemi~\cite{CGH68}. 
However, when the girth of a planar graph is sufficiently large, one can not only $3$-color, but $2$-color the vertices of the graph such that monochromatic components are short paths.
Borodin, Kostochka, and Yancey~\cite{BKY13} proved that the vertices of each planar graph of girth at least $7$ can be $2$-colored so that each monochromatic component has at most $2$ vertices, i.e., is a path of length at most $1$.
Note that the order of monochromatic components can not be decreased to $1$ as long as the graph is not bipartite.
In an earlier paper, Borodin and Ivanova~\cite{BI11} conjectured that any planar graph of girth $6$ can be $2$-colored such that each monochromatic component is a path of length at most $2$.

Here, we show that planar graphs of girth at least $6$ can be $2$-colored such that each monochromatic component is a path of length at most $14$.
Moreover, we prove a list version of this result.
On the other hand, for each positive integer $t\geq 3$, we construct a planar graph of girth $4$ such that in any coloring of vertices in $2$ colors there is a monochromatic path of length at least $t$. 

It remains open whether one can $2$-color the vertices of a planar graph of girth $5$ such that each monochromatic component is a short path.

Note that the problem we consider is a problem of {\it strong linear arboricity} or a {\it $k$-path chromatic number} introduced by Borodin \textit{et al.}~\cite{BI11} and Akiyawa \textit{et al.}~\cite{AEGW89} respectively.
Here, a linear arboricity of a graph is the smallest number of parts in a vertex-partition of the graph such that each part induces a forest with path components.
The $k$-strong linear arboricity or {\it $k$-path chromatic number} is the smallest number of colors in a vertex-coloring of the graph such that each monochromatic component is a path on at most $k$ vertices.

Let $L$ be a color list assignment for vertices of a graph $G$, i.e., $L: V(G) \rightarrow 2^{\mathbb Z}$.
We say that $c$ is an $L$-coloring if $c: V\rightarrow \mathbb{Z}$ such that $c(v)\in L(v)$ for each $v\in V(G)$. 

We prove the following theorems.

\begin{theorem}\label{thm:girth-6}
 For any planar graph of girth at least $6$ and any list assignment $L$ with lists of size $2$
 there is an $L$-coloring so that each monochromatic component is a path of length at most $14$.
\end{theorem}

\begin{theorem}\label{thm:girth-4-LB}
 For every positive integer $t$ there is a planar graph $G_t$ of girth $4$ such that any vertex coloring of $G_t$ in two colors results in a monochromatic path of length $t-1$.
 \end{theorem}

Our results are a contribution to the lively and active field of \emph{improper vertex colorings of planar graphs}, where the number of colors is strictly less than $4$ but various restrictions on the monochromatic components are imposed. 
For standard graph theoretic notions used here, we refer to~\cite{Die}.

\paragraph{Organization of the paper.}
In Section~\ref{sec:improper-colorings} we give a short survey of improper colorings of planar graphs, explain the relation to our results in the present paper, and point out some open problems.
In Section~\ref{sec:girth-6} we prove Theorem~\ref{thm:girth-6} and in Section~\ref{sec:girth-4} we prove Theorem~\ref{thm:girth-4-LB}.
We conclude with some open questions in Section~\ref{sec:conclusions}.

\section{Improper Colorings of Planar Graphs}\label{sec:improper-colorings}
 
A proper vertex-coloring of a graph is a coloring in which each monochromatic component is a single vertex, or, equivalently, in which there are no two adjacent vertices of the same color.
In this paper, a $c$-coloring, $c \geq 1$, of a graph is a (not necessarily proper) vertex coloring using $c$ colors.
As every planar graph has a proper $4$-coloring, we focus here on $2$-colorings and $3$-colorings.
The most studied variants of improper colorings are {\it defective}, {\it fragmented} and {\it $P_k$-free colorings}.
A survey on the topic was done in the bachelor thesis of Pascal Weiner~\cite{Weiner}.
 
\paragraph{Defective colorings.}
For a non-negative integer $k$, a vertex coloring is called \emph{$k$-defective} if each monochromatic component has maximum degree at most $k$.
We define $k_d(g,c)$ to be the smallest $k$ such that every planar graph of girth at least $g$ admits a $k$-defective $c$-coloring.
Defective colorings were introduced in 1986 by Cowen, Cowen and Woodall~\cite{CCW86}, who showed that $k_d(3,3) = 2$, i.e., every planar graph admits a $3$-coloring in which every monochromatic component has maximum degree at most $2$.
In fact, there is a $3$-coloring of any planar graph in which every monochromatic component is a path~\cite{Poh90}.
Eaton and Hull~\cite{EH99}, and independently {\v{S}}krekovski~\cite{Ske99}, proved that $k_d(3,2) = \infty$, i.e., there are planar graphs of girth $3$ for which any $2$-coloring results in a monochromatic component of arbitrarily high maximum degree.
Cowen, Goddard and Jerum~\cite{CGJ97} proved that every outerplanar graph admits a $2$-defective $2$-coloring.
Havet and Sereni~\cite{HS06} showed that for $c \geq 2, k \geq 0$ every graph of maximum average degree less than $c + \frac{ck}{c+k}$ admits a $k$-defective $c$-coloring.
By Euler's formula a planar graph of girth $g$ has maximum average degree less than $\frac{2g}{g-2}$.
Hence, the last result implies that $k_d(5,2) \leq 4$ and $k_d(6,2) \leq 2$.
The result of Borodin, Kostochka and Yancey~\cite{BKY13} shows that $k_d(7,2) = 1$.

\paragraph{Fragmented colorings.}
A $c$-coloring is \emph{$k$-fragmented} if each monochromatic component has at most $k$ vertices, and $k_f(g,c)$ denotes the smallest $k$ such that every planar graph of girth at least $g$ admits a $k$-fragmented $c$-coloring.
Fragmented coloring were first introduced in 1997 by Kleinberg~\textit{et~al.} in~\cite{KMRV97}, where they showed that $k_f(3,3) = \infty$, i.e., there is no $k$ such that every planar graph admits a $k$-fragmented $3$-coloring, a result that has been independently proven by Alon~\textit{et al.}~\cite{ADOV03}.
Esperet and Joret~\cite{EJ13} recently proved that $k_f(4,2) = \infty$, although this already follows from the fact that $k_d(4,2) = \infty$~\cite{Ske99}.

\paragraph{$P_k$-free colorings.}
Finally, a $c$-coloring is \emph{$P_k$-free} if there is no monochromatic path on $k$ vertices, and $k_p(g,c)$ denotes the smallest $k$ such that every planar graph of girth at least $g$ admits a $P_k$-free $c$-coloring.
Such $P_k$-free colorings were already introduced in 1968 by Chartrand, Geller and Hedetniemi~\cite{CGH68}, who showed that $k_p(3,3) = \infty$, i.e., there is no $k$ such that every planar graph admits a $P_k$-free $3$-coloring.
In a different paper~\cite{CGH71}, the same authors showed that same holds for outerplanar graphs and $2$ colors.
More than $20$ years later, the former result has been reproved by Akiyama~\textit{et al.}~\cite{AEGW89}, as well as Berman and Paul~\cite{BP89}.


We summarize the results for defective, fragmented and $P_k$-free colorings  using $2$ colors.
\begin{table}[htb]
 \centering
 \renewcommand{\arraystretch}{1.5}
 \newcolumntype{C}{>{\centering\arraybackslash}X}
 \newcolumntype{R}{>{\raggedleft\arraybackslash}Y}
 \begin{tabularx}{\textwidth}{|r|C|C|C|C|C|}
  \hline
   girth $g$ & $3$ & $4$ & $5$ & $6$ & $7$\\
  \hline
   \multirow{2}{*}{$k_d(g,2)$}
   & $\infty$ & $\infty$ & $\geq 2$ Fig.~\ref{fig:girth-5-example} & $\leq 2$ & $1$ \\
   & \cite{EH99} & \cite{Ske99} & $\leq 4$~~\cite{HS06}~~ & \cite{HS06} & \cite{BKY13} \\
  \hline
   \multirow{2}{*}{$k_f(g,2)$}
   & $\infty$ & $\infty$ & $\geq 3$ & $\leq 15$ & $2$ \\
   & \cite{KMRV97,ADOV03} & \cite{EJ13} &  Fig.~\ref{fig:girth-5-example} &  Thm.~\ref{thm:girth-6} & \cite{BKY13} \\
  \hline
   \multirow{2}{*}{$k_p(g,2)$}
   & $\infty$ & $\infty$ & $\geq 4$ & $\leq 16$ & $3$ \\
   & \cite{AEGW89,CGH68,BP89} & Thm.~\ref{thm:girth-4-LB} & Fig.~\ref{fig:girth-5-example} & Thm.~\ref{thm:girth-6} & \cite{BKY13} \\
\hline
 \end{tabularx}
 \caption{Improper $2$-coloring results for planar graphs of girth $g$.}
 \label{tab:overview}
\end{table}

\noindent
Theorem~\ref{thm:girth-6} and Theorem~\ref{thm:girth-4-LB} immediately imply the following, c.f. Table~\ref{tab:overview}.

\begin{corollary}
 We have that $k_f(6,2) \leq 15$, $k_p(6,2) \leq 16$, and $k_p(4,2) = \infty$.
\end{corollary}

Let us also mention that defective and fragmented colorings have also been considered for non-planar graphs of bounded maximum degree~\cite{HST03,ADOV03,BS09}, bounded number of vertices~\cite{LMST08}, and for minor-free graphs~\cite{WW09}.
In natural generalizations one allows different color classes to have different defect (see for example~\cite{BS07,MO15}), or considers list-coloring, which in fact, is the case in many of the results above.


\section{Proof of Theorem \ref{thm:girth-6}}\label{sec:girth-6}

For a list assignment $L$, we call an $L$-coloring of a planar graph {\it good} if each monochromatic component is a path of length at most $14$.
Throughout this section we let, for the sake of contradiction, a graph $G$ be a counterexample to Theorem~\ref{thm:girth-6}, so that $G$ is vertex-minimal, and among all such graphs has the largest number of edges.
I.e., $G$ has no good $L$-coloring, $G$ has a fixed plane embedding such that the addition of any edge to $G$ creates a crossing or a cycle of length at most $5$, and any subgraph of $G$ with fewer vertices has a good $L$-coloring.
To avoid a special treatment of an outer face we assume $G$ to be embedded without crossings on the sphere.
Note also that if a graph has no good $L$-coloring, then any of its supergraphs has no good $L$-coloring.

\paragraph{Idea of the proof.}
Our proof extends the ideas of Havet and Sereni~\cite{HS06}.
We start by proving some structural properties of $G$, i.e., that $G$ has minimum degree $2$, all faces of $G$ are chordless cycles of length at most $9$, and proving  a statement about the distribution of vertices of degree $2$ around every face $F$ in $G$.

If $G$ has a path $P$ of length at most $14$ with endpoints of degree $2$ and all inner vertices of degree $3$, then each vertex in $P$ has exactly one neighbor not in $P$.
Deleting the vertices of $P$ from $G$ gives a graph that has a good coloring.
Color each vertex of $P$ with a color different from the color of its neighbor not in $P$.
This gives a good coloring of $G$ contradicting the fact that $G$ is a minimal counterexample.

We generalize this simple argument, that uses a single path, to path systems, that is, sets of (directed) facial paths in $G$ with all inner vertices of degree $3$.
Next, we consider a charge of $\deg(v)-3$ at every vertex $v$ and define discharging rules shifting a charge of $1/2$ from the out-endpoint of every path in $X_0$ to its in-endpoint,  based on a specific path system $X_0$.
The total charge on all the vertices, before as well as after the discharging, is negative, giving some  vertices ending up with negative charge.
We consider such a vertex $w_0$, build another path system based on what is ''outgoing`` from this vertex, and show that the corresponding subgraph of $G$ is a reducible configuration.
Here, a subgraph $H$ is reducible if any good $L$-coloring of $G - V(H)$ (which exists by the minimality of $G$) could be extended to a good $L$-coloring of the whole graph $G$.
This contradicts the assumption that $G$ is a counterexample and hence concludes the proof.


\paragraph{Structural properties of $G$.}

\begin{lemma}\label{lem:min-degree}
 $G$ is connected and has minimum degree at least $2$.
\end{lemma}
\begin{proof}
 Indeed, if $G$ has a vertex $v$ of degree $1$, then a good coloring of $G - v$ can be extended to a good coloring of $G$ by choosing the color of $v$ to be different from its neighbor in $G - v$.
 If $G$ is not connected, then one of its connected components is a smaller counterexample, contradicting the definition of $G$.
\end{proof}

\begin{lemma} \label{lem:face-length}
 The boundary of each face of $G$ forms a chordless cycle of length at most $9$.
\end{lemma}
\begin{proof}
 First assume for the sake of contradiction there is a face $F$ whose closed boundary walk $W = u_0,\ldots,u_m$ is not a cycle.
 Then there is a vertex $u$ appearing at least twice on $W$, say $u=u_0 = u_j$ with $j \neq 0$.
 As $G$ has minimum degree $2$, each of the closed walks $W_1 = u_0,u_1,\ldots,u_j$ and $W_2 = u_j,u_{j+1},\ldots,u_0$ contains at least one cycle, i.e., has at least $6$ vertices.
 Note that the vertices $u_2$ from $W_1$ and $u_{j+3}$ from $W_2$ lie in distinct connected components of $G - u$.
 Moreover, as $G$ has minimum degree $2$, $u_2$ and $u_{j+3}$ are at distance $2$ and $3$ from $u$ along $W$, respectively.
 Hence any $u_2-u_{j+3}$ path goes through $u$ and, as there are no cycles of length at most $5$, the distance between $u_2$ and $u_{j+3}$ in $G$ is $5$.
 Thus we can add an edge $u_2u_{j+3}$ into $F$, creating a planar graph with girth at least $6$.
 A contradiction to edge-maximality of $G$.
 
 Thus, the boundary of each face $F$ forms a cycle, $C=u_0, \ldots, u_m, u_0$. 
 Assume that $C$ has length at least $10$, i.e., that $m\geq 9$. 
 Recall, that an \emph{ear} $E$ of a cycle $C$ is a path that shares only its endpoints with the vertex set of the cycle.
 For $i=0,\ldots,m$ let $G'(i)$ be obtained from $G$ by adding an edge $u_i,u_{i+5}$ into the face $F$, addition of indices modulo $m+1$. If $G'(i)$ has girth at least $6$, this contradicts the edge-maximality of $G$.
 So, there is a cycle on at most $5$ vertices containing edge $u_iu_{i+5}$ in $G'(i)$, denote a  shortest $u_i-u_{i+5}$ path in $G$ by $P(i,i+5)$.
 Its length is at most $4$, less than the distance between $i$ and $i+5$ along $C$ (as $m \geq 9$), so there is an ear of length $\ell$, $\ell \leq 4$, and $\ell$ is less than the distance between its endpoints along $C$. 
 The \emph{width} of an ear is the smallest distance between its endpoints along the cycle.
 If $Q$ is a path or a cycle and $P$ is a path in $Q$ with endpoints $u$ and $v$, we write $P = uQv$.
 We denote the length of $P$ as $||P||$.
 A \emph{$k$-ear} is an ear of length $k$.
 
 \medskip 
 
 {\it Case 1.} $C$ has a chord. \\
 Assume that $u_0u_k$ is a chord, $k\geq 5$. 
 A path $P=P(-3,2)$ must contain $u_0$ or $u_k$. 
 If $P$ contains $u_0$, then $||u_{-3}Pu_0|| \geq 3$ and $||u_0Pu_2|| \geq 2$, as otherwise $P \cup C$ contains a cycle of length at most $5$.
 Similarly, if $P$ contains $u_k$, then $||u_{-3}Pu_k|| \geq 2$ and $||u_kPu_2|| \geq 3$.
 In any case we have that $||P|| \geq 5$, a contradiction. See Figure~\ref{fig:cases1+2} left.
 
 \begin{figure}[htb]
 \centering
 \includegraphics{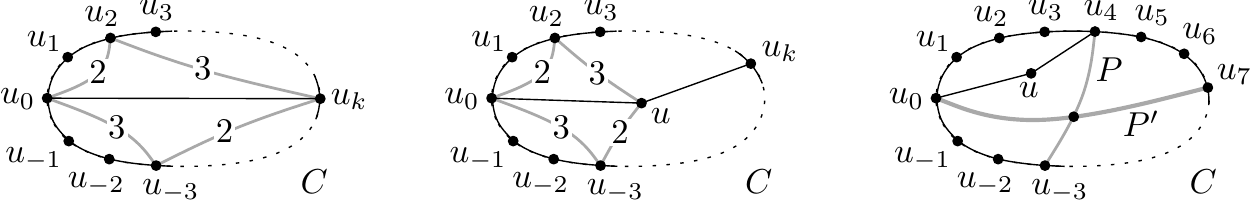}
 \caption{Illustration of Case~1 (left), Case~2 with $w \in \{u_0,u\}$ (middle) and Case~2 with $w=u_k$ (right). The face $F$ bounded by $C$ is shown as the outer face. The numbers indicate the minimum length of a path between the corresponding vertices.}
 \label{fig:cases1+2}
\end{figure}

{\it Case 2.} $C$ has an  ear of length $2$ and no chords. \\
Let $E$ be  a $2$-ear of smallest width, with vertices $u_0,u,u_k$, $4\leq k\leq (m+1)/2$.   A path $P=P(-3, 2)$  contains  $w \in \{u_0, u, u_k\}$.
If $w=u_0$,  see Figure \ref{fig:cases1+2} center,  then (as in Case~1) $||u_{-3}Pu_0|| \geq 3$ and $||u_0Pu_2|| \geq 2$, and if $w=u$, then $ ||u_2Pu||\geq 3$ and $||uPu_{-3}|| \geq 2$, as otherwise there is a cycle of length at most $5$ in $P \cup C$.
In both cases we have $||P|| \geq 5$, a contradiction.
So $w=u_k$,  see Figure \ref{fig:cases1+2} right,  $||wPu_2||\geq 2$, and $||u_{-3}Pw||\geq 2$, otherwise there is a chord.
Thus each of these segments has length $2$.
Since $||u_kPu_2||= 2$, $u_kPu_2$ is a subpath of $C$, otherwise there is a $2$-ear of a smaller width.
So $k=4$.   
Since $||u_{-3}Pu_k||= 2$ and $m \geq 9$, $||u_{-3} C u_k|| \geq 4$, so $||C||\geq 11$.      
Looking at $E$ in the other direction along $C$,  and taking $P' = P(2,7)$, 
we see symmetrically that $u_0P' u_7$ is  an ear of length $2$, that together with $E$ and $P$ creates a cycle of length $4$.

 \begin{figure}[htb]
 \centering
 \includegraphics{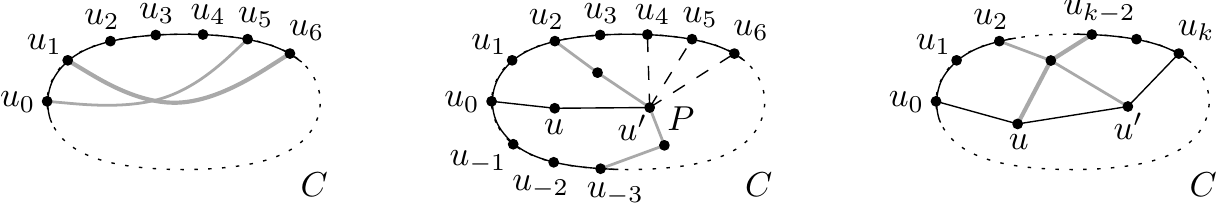}
 \caption{Illustration of Case~3 with only $4$-ears (left) and Case~3 with a $3$-ear (middle and right). The face $F$ bounded by $C$ is shown as the outer face.}
 \label{fig:case3}
\end{figure}

{\it  Case 3.}  $C$ has no  ears of length $2$ or chords.\\
Since each $P(i,i+5)$ results in an ear whose length is smaller than its width,
we see that either there is a  $u_i$-$u_{i+5}$ ear of length $4$ or an ear of length $3$ with width between $4$ and $6$.
If all such ears are of length $4$, then a $u_0-u_5$ ear and $u_1-u_6$ ear intersect and form, together with $C$, a cycle of length at most $5$, a contradiction, see Figure \ref{fig:case3} left.
Assume that $E$ is a  $3$-ear   $u_0, u, u',  u_k$, $4\leq  k \leq 6$,  and all other  $3$-ears  have width either at most  $3$  or at least $k$, see Figure \ref{fig:case3} center and right.
A path  $P=P(2, -3)$ contains a $3$- or a $4$-ear.  Let $w$ be a point on $P$ and $E$. We have that $||u_{-3}P w||\geq 2$ and $||wPu_2||\geq 2$, otherwise
there is either a chord, a $2$-ear, or a cycle of length at most $5$.
It follows that $||u_{-3}Pw|| = 2 = ||wPu_2||$.
Then $w= u'$, and $k\geq 5$.  Looking at $E$ in the other direction, we see symmetrically,
that there is a path of length $2$ between $u_{k-2}$ and $u$, implying the existence of a triangle containing $u$ and $u'$, a contradiction.

\medskip

Thus $C$ has length at most $9$. If $C$ has  a chord, then 
there is a cycle of length at most $5$, a contradiction. So, $C$ is a chordless cycle.
This concludes the proof the Lemma.
\end{proof}


\begin{lemma}\label{lem:case1}
 Let $F$ be any face of $G$ incident to a vertex of degree $2$.
 Then $F$ is incident to a vertex of degree at least $4$, and if $F$ is incident to at least two vertices of degree $2$, then there is a vertex of degree at least $4$ between any two such vertices on both paths along $F$.
\end{lemma}
\begin{proof}
 Let $C$ be the simple chordless cycle bounding $F$.
 First, assume for the sake of contradiction that $C$ contains exactly one vertex $v$ of degree $2$ and all other vertices of degree $3$.
 Consider a good $L$-coloring of $G'= G - V(C)$ and give each vertex $u$ of $C$ of degree $3$ a color in $L(u)$ different from the color of its neighbor in $G'$.
 Give $v$ a color in $L(v)$ such that $C$ does not form a monochromatic cycle.
 As a result, the set of monochromatic components of $G$ is formed by the monochromatic components of $G'$, and paths on at most $8$ vertices formed by vertices of $C$.
   
 Second, assume that $C$ contains two vertices $u,v$ of degree $2$ and a $u-v$ path $P$ in $C$ has no inner vertices or only inner vertices of degree $3$.
 Consider a good $L$-coloring of $G'= G - V(P)$ and give the vertices of $P$ colors from their lists, different from the colors of their unique neighbors in $G'$.
 This does not extend any connected monochromatic component of $G'$ and every new monochromatic component is contained in $P$, i.e., a path on at most $8$ vertices.
 
 I.e., in both cases we have found a good $L$-coloring of $G$, a contradiction to $G$ being a counterexample.
\end{proof}

\paragraph{Path systems.}

A \emph{path system} is a set $X$ of (not necessarily edge-disjoint) directed facial paths in $G$ with all inner vertices being of degree $3$, such that no vertex is an endpoint of one path in $X$ and an inner vertex of another path in $X$.
For a path $P \in X$ directed from vertex $u$ to vertex $v$, we call $u$ the \emph{out-endvertex} and $v$ the \emph{in-endvertex} of $P$.
For a path system $X$, the vertices that are the in-endvertices or out-endvertices of some path in $X$ are called the \emph{endvertices of $X$}, while the \emph{inner vertices of $X$} are the inner vertices of some path in $X$.
For any vertex $v$ in $G$ let $\outdeg_X(v)$ and $\indeg_X(v)$ denote the number of paths in $X$ with out-endvertex $v$ and in-endvertex $v$, respectively.
Note that for an inner vertex $v$ of $X$ we have $\outdeg(v) = \indeg(v) = 0$.
A directed path $P$ is \emph{occupied} by a path system $X$ if the first or last edge of $P$ (incident to its out-endvertex or in-endvertex) is contained in some path in $X$.
So, if $P \in X$, then $P$ is occupied by $X$.
Let us emphasize that throughout the paper $\deg(v)$ and $N(v)$ always refer to the degree and neighborhood of vertex $v$ in $G$, even when we consider other subgraphs of $G$ later.

For a path system $X$ and any two vertices $u,v$ in $G$ we say that \emph{$u$ reaches $v$ in $X$}, denoted by $u \to_X v$, if there is a sequence $u = v_1,\ldots,v_k = v$ of vertices and a sequence $P_1,\ldots, P_{k-1}$ of paths in $X$ such that $v_i$ and $v_{i+1}$ are out-endvertex and in-endvertex of $P_i$, respectively, $i=1,\ldots,k-1$.
Then $X$ is \emph{acyclic} if there are no two distinct vertices $u,v$ with $u \to_X v$ and $v \to_X u$.
For a vertex $w$ of $G$, we define $X^+(w) \subseteq X$ to be the path system consisting of all paths in $X$ whose out-endvertex is $w$ or reachable from $w$ in $X$.

A path system $X$ is \emph{nice} if each of the following properties~\ref{enum:bidirected}--\ref{enum:degree-5} holds.
A path system $X$ with a distinguished vertex $r$, called root, is \emph{almost nice} if the properties~\ref{enum:bidirected}--\ref{enum:degree-5} hold for all vertices different from $r$.

See Figure~\ref{fig:nice-properties} for an illustration.

\begin{enumerate}[label = \textbf{(D\arabic*)}, itemsep=0pt]
 \item Every edge that belongs to two paths in $X$ joins two vertices of degree $3$ each.\label{enum:bidirected}
 \item Every vertex of degree $2$ has outdegree $0$ in $X$.\label{enum:degree-2}
 \item Every vertex of degree $3$ has indegree $0$ and outdegree $0$ in $X$.\label{enum:degree-3}
 \item Every vertex of degree $4$ has positive indegree in $X$ only if it has outdegree $3$ in $X$.\label{enum:degree-4}
 \item Every vertex of degree at least $5$ has in-degree $0$ in $X$.\label{enum:degree-5}
\end{enumerate}

\begin{figure}[htb]
 \centering
 \includegraphics{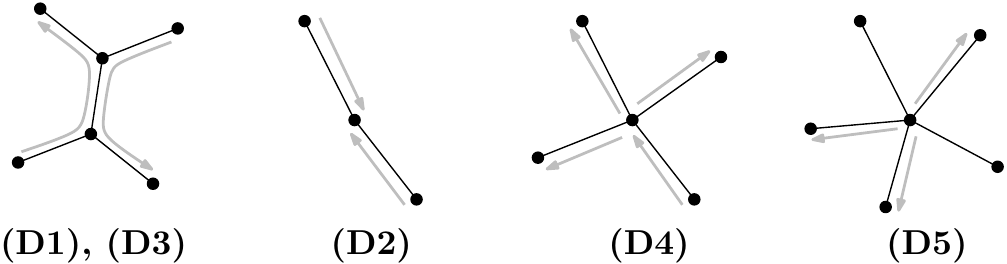}
 \caption{Illustration of properties~\ref{enum:bidirected}--\ref{enum:degree-5}.}
 \label{fig:nice-properties}
\end{figure}

%
The following statements follow immediately from the definitions above.

\begin{lemma}\label{lem:easy-properties}
 For every path system $X$ each of the following holds.
 \begin{enumerate}[itemsep = 0pt, label = (\arabic*)]
  \item If no path $P \in X$ is occupied by $X - \{P\}$, then $X$ satisfies~\ref{enum:bidirected}.\label{enum:easy-1}
  \item If $X' \subseteq X$ and $X$ satisfies any of~\ref{enum:bidirected}--\ref{enum:degree-3},~\ref{enum:degree-5}, then so does $X'$.\label{enum:easy-2}
  \item If $X' \subseteq X$ and $X$ is acyclic, then so is $X'$.\label{enum:easy-3}
  \item If $X$ is nice and $w$ is a vertex, then $X^+(w)$  with root $w$  is almost nice.\label{enum:easy-4}
 \end{enumerate}
\end{lemma}


\paragraph{Discharging with respect to a path system $X$.}

Given a path system $X$, consider the following discharging:
Put charge $ch(v) = \deg(v) -3$ on each vertex of $G$. 
Note that $ch(v) =-1$ for a vertex of degree $2$, and $ch(v)\geq 0$ for all other vertices. 
As all facial cycles have length at least $6$, we have $6f \geq 2e$, where $f$ denotes the number of faces of $G$.
Together with Euler's formula $n - e + f = 2$ this implies $n - e + e/3 \geq 2$.
Thus the total charge is $\sum_{v\in V(G)}  (\deg(v) - 3) =  2e - 3n \leq -6$.
  
Define $ch'(v) = ch(v) + \frac{1}{2}(\indeg_X(v) - \outdeg_X(v))$.
Intuitively, for every path in $X$ a $1/2$-charge is sent from out-endvertex to in-endvertex.
Thus, the total sum of charges in $ch'$ is the same as in $ch$, i.e., $\sum_v ch(v) = \sum_v ch'(v)$.

\paragraph{Defining a path system $\P$.}


As $G$ has girth at least $6$, there is a vertex $v$ of degree $2$ in $G$ and by Lemma~\ref{lem:case1} both faces incident to $v$ contain a vertex of degree at least $4$.
So there are faces with at least two vertices of degree different from $3$.
For each such face $F$ the boundary of $F$ can be uniquely partitioned into edge-disjoint counterclockwise oriented paths with all inner vertices of degree $3$ and endpoints of degree different from $3$.
We denote by $\P$ the path system consisting of all such paths with in-endvertex of degree $2$ or $4$ and out-endvertex of degree at least $4$, for all faces $F$ with at least two vertices of degree different from $3$.
So for each path in $\P$ the degrees $d_1,d_2$ of its in-endvertex and out-endvertex, respectively, satisfy $(d_1,d_2) \in \{(2,4),(2,\ell),(4,4),(4,\ell) \mid \ell \geq 5\}$.

By Lemma~\ref{lem:face-length} every face of $G$ is bounded by a simple chordless cycle of length at most $9$.
Thus every $P \in \P$ is a path on at most $8$ edges.
As any two paths in $\P$ in the boundary of the same face $F$ are edge-disjoint, every edge of $G$ lies in at most two paths in $\P$, at most one for each face incident to the edge.
If an edge lies in two paths in $\P$, these paths have the edge oriented in opposite directions.
For a vertex $v$ in $G$ with $\deg(v) = 3$ we have $\outdeg_{\P}(v) = \indeg_{\P}(v) = 0$ by definition. 
Note that by Lemma~\ref{lem:case1} for every vertex $v$ with $\deg(v) = 2$ we have $\outdeg_{\P}(v) = 0$ and $\indeg_{\P}(v) = 2$. 
For a vertex $v$ with $\deg(v) \geq 5$ we have $\indeg_{\P}(v)=0$, i.e., $\P$ has properties~\ref{enum:degree-2},~\ref{enum:degree-3} and~\ref{enum:degree-5}.
We provide an example illustrating these concepts in Figure~\ref{fig:big-example}.

\begin{figure}[htb]
 \centering
 \includegraphics{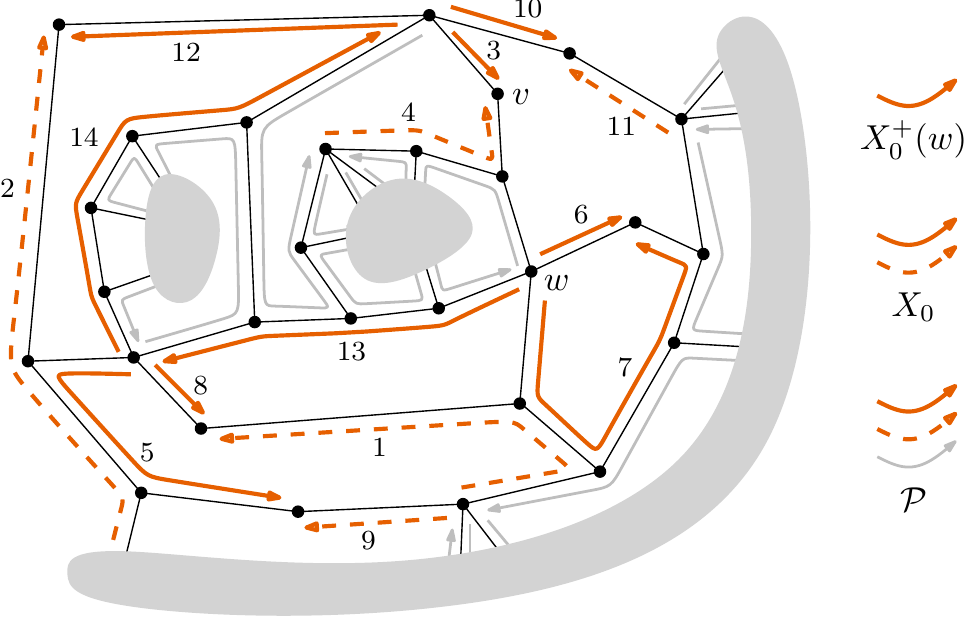}
 \caption{A (part of a) planar graph of girth $6$ and the path systems in $\P$, $X_0$ and $X_0^+(w)$. The labels show the order in which paths of $X_0$ were selected. Paths $13$ and $14$ were selected in step~\ref{enum:step-2}.}
 \label{fig:big-example}
\end{figure}


\paragraph{Defining a path system $X_0 \subseteq \P$.}

We define $X_0 \subseteq \P$ selecting paths one by one, using the following procedure, where we go through the vertices in question in an arbitrary but fixed order.
At all times, let $X_0$ denote the set of already chosen paths, initially $X_0 = \emptyset$.

\begin{enumerate}[label = \arabic*.)]
 \item For every vertex $v$ with $\deg(v) = 2$ we put a path from $\P$ into $X_0$ if its in-endpoint is $v$ and if it is not occupied by $X_0$.\label{enum:step-1}
\end{enumerate}
After step~\ref{enum:step-1} is done for all vertices of degree~$2$, we proceed as follows.
\begin{enumerate}[label = \arabic*.), start = 2]
 \item For every vertex $v$ with $\deg(v) = 4$ and $\outdeg_{X_0}(v) = 3$, put a path from $\P$ into $X_0$ if its in-endpoint is $v$ and if it is not occupied by $X_0$.\label{enum:step-2}
\end{enumerate}

Later, we shall show that the final path system $X_0$ is nice and acyclic.
For now, we only need to observe that~\ref{enum:degree-2} is satisfied and $\indeg_{X_0}(u)= 2$ for every vertex $u$ of degree $2$. 
In fact,~\ref{enum:degree-2} holds for $\P$ and thus by Lemma~\ref{lem:easy-properties}~\ref{enum:easy-2} it also holds for $X_0\subseteq \P$.
Assume now that $\indeg_{X_0}(u) <2$.
I.e., $P$, one of the two paths in $\P$ with in-endvertex $v$ was occupied by during step~\ref{enum:step-1}.
As all in-endvertices of paths chosen in step~\ref{enum:step-1} are of degree $2$, and the out-endvertex of $P$ has degree at least $4$, this is impossible.


\paragraph{Defining the vertex $w_0$ based on discharging with respect to $X_0$.}

Let us apply discharging to $X_0$.
For every vertex $u$ with $\deg(u) = k$, we have $\indeg_{X_0}(u) \geq 0$ and $\outdeg_{X_0}(u) \leq k$, i.e., $u$ looses a charge of at most $\frac{k}{2}$.
Thus if $\deg(u) = k \geq 6$, the remaining charge $ch'(u)$ is at least $k -3 - \frac{k}{2} \geq 0$.
If $\deg(u) = 3$, then $\outdeg_{X_0}(u) = \indeg_{X_0}(u) = 0$ and hence $ch(u) = ch'(u) = 0$.
If $\deg(u) = 2$,  then $\indeg_{X_0}(u) = 2$ and $\outdeg_{X_0}(u) = 0$ and hence $ch'(u) = \deg(u) - 3 + \frac{1}{2}(2 - 0) = 0$.

On the other hand we have $\sum_v ch(v) = \sum_v ch'(v)$.
As $\sum_v ch'(v) \leq -6$ there is a vertex $w_0$ in $G$ with $ch'(w_0) < 0$.
With the above considerations we conclude that $\deg(w_0) \in \{4,5\}$.

If $\deg(w_0)=5$, then $0> ch'(w_0) \geq (5-3) - \frac{1}{2} \outdeg_{X_0}(w_0)$, so $\outdeg_{X_0}(w_0) \geq 5$.
Since $\outdeg_{X_0}(w_0) \leq \deg(w_0)$, we have that $\outdeg_{X_0}(w_0) =5$. 
If $\deg(w_0)=4$, then $0> ch'(w_0) = (4-3) + \frac{1}{2}( \indeg_{X_0}(w_0)-\outdeg_{X_0}(w_0))$, so either $\outdeg_{X_0}(w_0) =4$ or ($\outdeg_{X_0}(w_0) =3 $ and $\indeg_{X_0}(w_0) = 0$).
In particular, exactly one of the following must hold for the vertex $w_0$ with $ch'(w_0) < 0$:

\begin{enumerate}[label = \textbf{Case \arabic*:}, leftmargin=4em, itemsep=0pt]
 \item $\deg(w_0) \in \{4,5\}$ and $\outdeg_{X_0}(w_0) = \deg(w_0)$.\label{enum:case1}
 \item $\deg(w_0) = 4$, $\outdeg_{X_0}(w_0) = 3$ and $\indeg_{X_0}(w_0) = 0$.\label{enum:case2}
\end{enumerate}

For example, in Figure~\ref{fig:big-example} we see that \textbf{Case~2} applies to vertex $w$.

%
%


\paragraph{Defining rooted path systems $X_1,X_2,X_3,X_4$ based on $w_0$ and $X_0$.}

Depending on the structure of $w_0$ and $X_0$ we shall define one of four path systems $X_1,X_2,X_3,X_4$, each $X_i$ with a specified vertex $w_i$, called the \emph{root}, $i=1,2,3,4$.
Path systems $X_1,X_3$ will be chosen as subsystems of $X_0$, $X_2$ as a subsystem of $X_0$ together with an additional path from $\P$, and $X_4$ as a subsystem of $X_0$ together with a subpath of a path from $\P$.
Note that each of $X_1,X_2,X_3,X_4$ consists of paths of length at most $8$.

\medskip

\textbf{Case 1:} $\deg(w_0) \in \{4,5\}$ and $\outdeg_{X_0}(w_0) = \deg(w_0)$.\\
In this case we define $X_1 = X_0^+(w_0)$ with root $w_0$.

\medskip

\textbf{Case 2:} $\deg(w_0) = 4$, $\outdeg_{X_0}(w_0) = 3$ and $\indeg_{X_0}(w_0) = 0$.\\
Consider the unique edge $e$ at $w_0$ not contained in any path in $X_0$.
As $w_0$ has outdegree $3$, the clockwise next edge $e'$ at $w_0$ after $e$ is contained in some path in $X_0$ with out-endvertex $w_0$.
The in-endvertex of this path is in the face $F$ incident to $w_0$, $e$, and $e'$.
See the middle part of Figure~\ref{fig:our-configurations} for an illustration.
So $F$ has at least two vertices of degree different from $3$.
Thus its boundary contains a counterclockwise path $P$ with in-endvertex $w_0$, using the edge $e$, all inner vertices of degree $3$ or no inner vertices at all and out-endvertex $v$ with $\deg(v) \neq 3$.
Let $e''$ be the edge of $P$ incident to $v$.
If $\deg(v)=2$, then from the definition of $X_0$, $\indeg_{X_0}(v) = 2$.
Thus $e''$ belongs to a path in $X_0$ with in-endpoint $v$.
If $\deg(v)\geq 4$ then $P \in \P$, and in step~\ref{enum:step-2} of the construction of $X_0$ the path $P$ must have been rejected because it was occupied, i.e., $e''$ is contained in another path in $X_0$.
As $P$ has $v$ as out-endvertex, the other path has $v$ as in-endvertex and hence $\deg(v) = 4$.
So, $\deg(v)\in \{2,4\}$ and $e''$ lies in some path in $X_0$ with in-endvertex $v$.
In particular it follows that $e \neq e''$, i.e., $P$ has at least one inner vertex.

%


\begin{figure}[htb]
 \centering
 \includegraphics{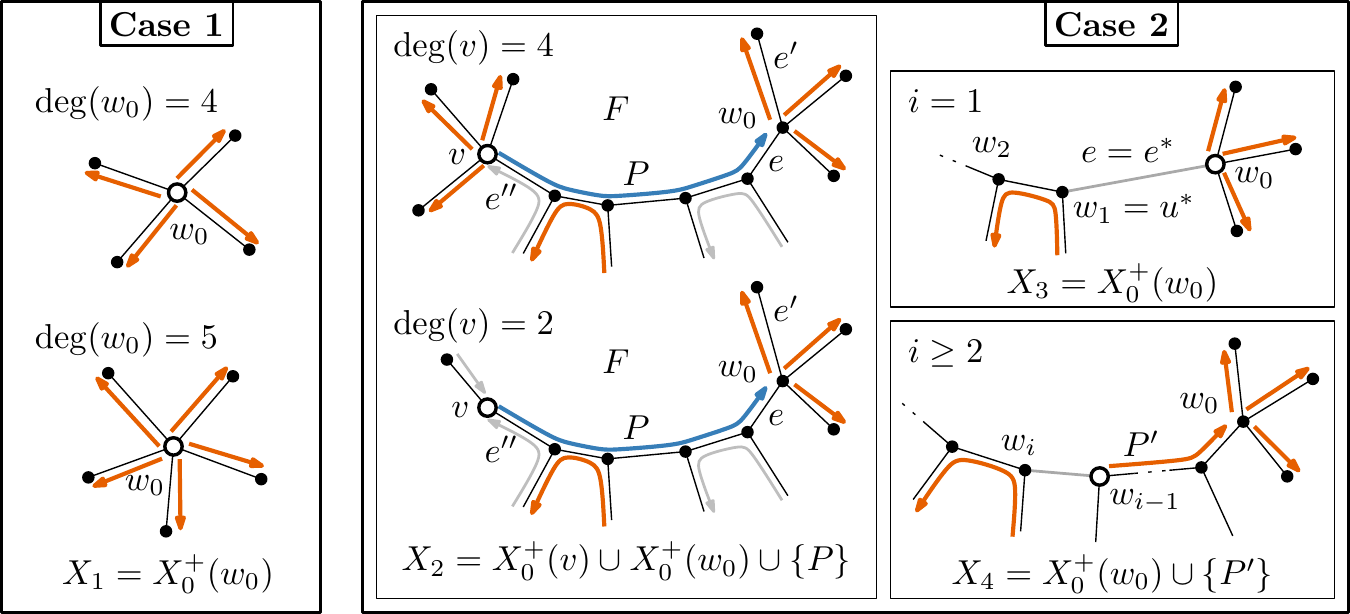}
 \caption{Illustrations of the rooted path systems $X_1,X_2,X_3,X_4$ with highlighted roots.}
 \label{fig:our-configurations}
\end{figure}


Next, we distinguish the cases when $v$ is reachable from $w_0$ in $X_0$ or not, corresponding to the right and middle part of Figure~\ref{fig:our-configurations}, respectively.
In case $v$ is not reachable from $w_0$ in $X_0$, we define $X_2 = X_0^+(v) \cup X_0^+(w_0) \cup \{P\}$ with root $v$.
 
When $v$ is reachable from $w_0$ in $X_0$, let $w_0,w_1,\ldots,w_{k-1},w_k=v$, $k \geq 2$, denote the vertices of $P$ in their order along $P$ from its in-endvertex $w_0$ to its out-endvertex $v$.
Recall that $P$ has at least one inner vertex.
Let $i$ be the smallest index such that $w_i \neq w_0$ and $w_i$ is contained in a path in $X_0^+(w_0)$.
See the right part of Figure~\ref{fig:our-configurations}.
As $v = w_k$ is reachable from $w_0$ in $X_0$, this index is well-defined.
If $i=1$, we define $X_3 = X_0^+(w_0)$ with root $w_0$.
Otherwise we denote the directed $w_{i-1}$-to-$w_0$ subpath of $P$ by $P'$ and define $X_4 = X_0^+(w_0) \cup \{P'\}$ with root $w_{i-1}$.
This is for example the case for vertex $w_0 = w$ in Figure~\ref{fig:big-example}.

\begin{lemma}\label{lem:our-configurations}
~
 \begin{enumerate}[label = (\roman*), itemsep = 0pt]
  \item Each of $X_0,X_1,X_2,X_3,X_4$ is acyclic.\label{enum:acyclic}
  \item $X_0$ and $X_1$ are nice.\label{enum:X_1-nice}
  \item If the root $v$ of $X_2$ has degree $4$, then $X_2$ is nice.\label{enum:X_2-nice}
  \item If the root $v$ of $X_2$ has degree $2$, then $X_2$ is almost nice with $\outdeg_{X_2}(v)=1$.\label{enum:X_2-almost-nice}
  \item $X_3$ is almost nice with $\outdeg_{X_3}(w_0) = 3$ and $\indeg_{X_3}(w_0) = 0$.\label{enum:X_3-almost-nice}
  \item $X_4$ is almost nice with $\outdeg_{X_4}(r)=1$ and $\indeg_{X_4}(r)=0$ for the root $r$ of $X_4$.\label{enum:X_4-almost-nice}
  \item If $j \in \{1, 2, 3, 4\}$ then each endvertex of $X_j$, different from the root, has degree $2$ or $4$ in $G$, the root has degree $2, 3, 4$, or $5$, and each inner vertex of $X_j$ has degree $3$.
  Moreover, each vertex of $X_j$ has at most one neighbor that is not in $X_j$.\label{enum:overall-degrees}
 \end{enumerate}
\end{lemma}
\begin{proof}{\ \\}
 \ref{enum:acyclic}:
 First, we shall show that $X_0$ is acyclic.
 Assume for the sake of contradiction that $v_0,\ldots,v_{k-1}$ and $P_0,\ldots,P_{k-1}$ are two sequences of vertices and paths in $X_0$ such that for every $i \in \{0,\ldots,k-1\}$ we have that $v_i$ and $v_{i+1}$ are out-endvertex and in-endvertex of $P_i$, respectively (all indices modulo $k$).
 For each $i \in \{0,\ldots,k-1\}$ we have $\deg(v_i) \in \{2,4\}$, as we add only paths with such in-endvertices to $X_0$ in step~\ref{enum:step-1} and~\ref{enum:step-2}.
 Moreover, $v_i$ is out-endvertex of $P_{i-1}$ and thus we have that $\deg(v_i) \neq 2$.
 Hence for each $i \in \{0,\ldots,k-1\}$ we have $\deg(v_i) = 4$ and $P_i$ was put into $X_0$ in step~\ref{enum:step-2} because $v_{i+1}$ was the out-endvertex of exactly three already chosen paths.
 Assume without loss of generality that $P_0$ was the path that was put into $X_0$ in step~\ref{enum:step-2} first among the paths $P_0,\ldots,P_{k-1}$.
 This means that the path $P_1$, whose out-endvertex is $v_1$, was already put into $X_0$.
 This contradicts that $P_0$ was the first and proves that $X_0$ is acyclic.
 
 Now, $X_1,X_3 \subseteq X_0$ are acyclic by Lemma~\ref{lem:easy-properties}~\ref{enum:easy-3}.
 Moreover, $X_2 = X_0^+(v) \cup X_0^+(w_0) \cup \{P\}$ is acyclic, because $X_0^+(v),X_0^+(w_0) \subseteq X_0$, $P$ has in-endvertex $w_0$ and out-endvertex $v$, and (in this case) $v$ is not reachable from $w_0$ in $X_0$.
 Finally, $X_4 = X_0^+(w_0) \cup \{P'\}$ is acyclic, because $X_0^+(w_0) \subseteq X_0$ and $V(P') \cap \left(\bigcup_{P \in X_0} V(P)\right) = \{w_0\}$.
 
 \medskip
 
 \noindent
 \ref{enum:X_1-nice},~\ref{enum:X_3-almost-nice}:  
 Consider $X_0$.
 As mentioned earlier, $\P$ satisfies~\ref{enum:degree-2},~\ref{enum:degree-3} and~\ref{enum:degree-5} and by Lemma~\ref{lem:easy-properties}~\ref{enum:easy-2} so does $X_0$.
 Moreover, by definition $X_0$ satisfies~\ref{enum:bidirected} and~\ref{enum:degree-4}, thus $X_0$ is nice.    Consider $X_1$ and $X_3$.
 By Lemma~\ref{lem:easy-properties}~\ref{enum:easy-4} we have that $X_1$ and $X_3$ are almost nice, as both are defined as $X_0^+(w_0)$ with root $w_0$.
 By construction, $\outdeg_{X_3}(w_0) = 3$ and $\indeg_{X_3}(w_0) = 0$, which proves~\ref{enum:X_3-almost-nice}.
 For $X_1$ note that, if $\deg(w_0)=4$, then~\ref{enum:degree-4} holds for $w_0$ in $X_1$, and if $\deg(w_0)=5$, then~\ref{enum:degree-5} holds for $w_0$ since $X_1 \subseteq X_0 \subseteq \P$.
 Thus $X_1$ is nice.
  
 \medskip

 \noindent
 \ref{enum:X_2-nice},~\ref{enum:X_2-almost-nice}:
 Next consider $X_2 = X_0^+(v) \cup X_0^+(w_0) \cup \{P\}$ with root $v$ and path $P$ as defined above, see the middle part of Figure~\ref{fig:our-configurations}.
 Each of $X_0^+(v)$, $X_0^+(w_0)$ is almost nice by Lemma~\ref{lem:easy-properties}~\ref{enum:easy-4} and the niceness of $X_0$.
 We have neither $w_0 \to_{X_0} v$ (by assumption) nor $v \to_{X_0} w_0$ (as $\indeg_{X_0}(w_0) = 0$).
 So $X_0^+(v) \cup X_0^+(w_0)$ satisfies~\ref{enum:bidirected}--\ref{enum:degree-5}, except perhaps for $w_0$ and $v$.
 As $X_2$ additionally contains the path $P$ from $v$ to $w_0$, we have that~\ref{enum:degree-4} is satisfied for $w_0$ and thus $X_2$ is nice when $\deg(v) = 2$ and almost nice when $\deg(v) = 4$.
 Because $X_0$ is nice, i.e., satisfies~\ref{enum:degree-2}, we have $\outdeg_{X_0}(v) = 0$ when $\deg(v) = 2$.
 As $X_2 - P \subseteq X_0$ and $P$ is outgoing at $v$, we have $\outdeg_{X_2}(v) = 1$.
 
 \medskip
 
 \noindent
 \ref{enum:X_4-almost-nice}:
 The system $X_4 = X_0^+(w_0) \cup \{P'\}$ is almost nice, because $P'$ and $X_0^+(w_0)$ share only vertex $w_0$, $X_0^+(w_0)$ is almost nice by Lemma~\ref{lem:easy-properties}~\ref{enum:easy-4}, and $P'$ is incoming at $w_0$.
 
 \medskip
 
 \noindent
 \ref{enum:overall-degrees}:
 These properties are corollaries of the almost-niceness of $X$ and the considerations for the root in the previous items.
\end{proof}

\paragraph{Coloring reducible configurations based on $X_1,X_2,X_3,X_4$.}

Recall that a coloring is good if each monochromatic component is a path of length at most $14$.
A \emph{reducible configuration} is a non-empty subgraph $H$ of $G$, such that any good $L$-coloring of $G - V(H)$ (which exists by the minimality of $G$) can be extended to a good $L$-coloring of $G$ in which every edge between a vertex in $H$ and a vertex outside of $H$ is colored properly.
Showing that $G$ has a reducible configuration will conclude the proof of Theorem~\ref{thm:girth-6}.
For convenience we say that $X_i$ is reducible if the subgraph $H$ of $G$ induced by the vertices in $X_i$ is a reducible configuration, $i=1,2,3,4$.

\begin{lemma}\label{lem:reducible-general}
 Each of $X_1,X_2,X_3,X_4$ is reducible, whenever it is defined.
\end{lemma}
\begin{proof}
 Consider $j \in \{1,2,3,4\}$.
 Let $r$ be the root of $X_j$, $H$ be the subgraph of $G$ consisting of all vertices and undirected edges in the path system $X_j$.
 Let $V_1$ be the set of vertices of $H$ and $H'$ be the subgraph of $G$ induced by $V_1$, i.e., $H\subseteq H'$.
 Let $W \subseteq V_1$ be the set of endvertices of $X_j$.
 Recall, that by Lemma~\ref{lem:our-configurations}\ref{enum:overall-degrees}, if $w \in W - \{r\}$, then  $\deg(w) \in \{2,4\}$, $\deg(r) \in \{2,3,4,5\}$, and if $u \in V_1 - W$, then $\deg(u) = 3$.
 In addition, the niceness or almost niceness of $X_j$ and the degree conditions for $r$ given in Lemma~\ref{lem:our-configurations} any vertex from $V_1$ has at most one neighbor not in $V_1$ and each vertex in $V_1-\{r\}$ has at most one incident edge from $E(G) - E(H)$.
 In particular, $E(H') - E(H)$ is a matching, unless $j = 4$, in which case $E(H') - E(H)$ might contain two edges incident to $r$.
 In case $j \neq 3$, let $E_1 = E(H') - E(H)$.
 Otherwise (when $j = 3$) let $e^* = ru^*$ denote the unique edge in $E(H') - E(H)$ incident to the root $r$ and let $E_1 = E(H') - (E(H) \cup e^*)$.
 In Figure~\ref{fig:our-configurations} on the right we have $r = w_0$ and $u^* = w_1$.
 Note that if $j \in \{1,2\}$, then there are no edges from $E(H')-E(H)$ incident to $r$.

%

We shall be coloring different sets of vertices of $G$ one after another.

 \begin{itemize}[leftmargin=1em]
  \item First we make a good $L$-coloring $c'$ of $G - V_1$, which exists by the minimality of $G$.
  Note that $G - V_1$ might be empty.
 \end{itemize}
 

We shall color $V_1$ so that no vertex in $V_1$ has the same color as its neighbor (if exists) in $V(G)-V_1$ and such that each monochromatic path with vertices in $V_1$ is contained in the union of two paths from $X_j$.

 \begin{itemize}[leftmargin=1em]
  \item Consider $A \subseteq V_1$, the set of vertices that have a neighbor in $V(G)-V_1$.
  As $X_j$ satisfies~\ref{enum:degree-4} no vertex of degree $4$ in $H'$ is in $A$, except for possibly the root $r$.
  We color each vertex $v\in A$ such that its color is from $L(v)$ and differs from the color of its neighbor in $V(G)-V_1$.
 \end{itemize}

 Now, no matter how we color $V_1 - A$, each monochromatic path has all its vertices completely in $V_1$ or completely in $V(G)-V_1$.
 Since the coloring of $V(G)-V_1$ is good, each monochromatic path there has length at most $14$.
 So, we only need to color $V_1-A$ so that each monochromatic path with vertices in $V_1$ has length at most $14$.

 \begin{itemize}[leftmargin=1em]
  \item Consider the vertices of $E_1$.
   First assume $r \in V(E_1)$, this could be only if $j = 2$ or $j = 4$.
   If $r \in A$, then $r$ is already colored and if $r \notin A$, we give $r$ any color from its list.
   Next, we color every neighbor of $r$ in $E_1$ with a color from its respective list different from the color of $r$.
   Finally, we color the remaining vertices of $E_1$ from their lists such that each edge of $E_1$ has endpoints of different colors.
   If $r \not\in V(E_1)$, i.e., $E_1$ is a matching, color $V(E_1)$ such that each edge is colored properly. 
 \end{itemize}
 
 \noindent
 This ensures that eventually every monochromatic component of $H'$ is a subgraph of $H$ or $H \cup e^*$ in case $j = 3$.
 
 \begin{itemize}[leftmargin=1em]
  \item Consider the set $B$ of vertices from $V_1-A$ not incident to $E_1$ and of degree $3$.
   Note that $r \notin B$ because it is either of degree different from $3$ or is incident to $E_1$ in case when $j=3$.
   Hence $B$ consists only of inner vertices of $X_j$, i.e., $B = V_1 - (A \cup V(E_1) \cup W)$.
   For any $u \in B$ all three edges incident to $u$ are in $H$, so $u$ lies on at least two paths in $X_j$.
   We consider the paths in $X_j$ in any order and when we process a path $P \in X_j$, we color the vertices in $B \cap V(P)$.
   For the current path $P$ and the current vertex $u \in B \cap V(P)$, consider the neighbor $u'$ of $u$ not in $P$.
   If $u'$ is not colored, color $u$ arbitrarily from its list.
   Otherwise, color $u$ with a color different from the color of $u'$.
 \end{itemize}
 
 \noindent
 This ensures that every monochromatic component of $H' - W$ is completely contained in some path in $X_j$.
 It remains to color the vertices in $W - A$ and in $e^* = ru^*$ (if it exists) in such a way that $e^*$ is not monochromatic and at most two monochromatic components of $H'-W$ are part of the same monochromatic component of $H'$.
 
 \begin{itemize}[leftmargin=1em]
  \item Consider the vertices in $W - A$ and the vertex $u^*$ (if it exists).
   Recall that $u^*$ is an inner vertex of some path in $X_j$ and hence $u^* \notin W$.
   For each $u \in W-A$, consider the paths in $X$ with in-endvertex $u$ and let $S(u)$ be the set of immediate neighbors of $u$ on those paths, i.e., $v \in S(u)$ if $uv \in E(H)$ and $u$ is the in-endvertex of the path in $X_j$ containing $uv$.
   In particular, $S(r) = \emptyset$, and for $u \neq r$ we have $|S(u)| = 1$ if $\deg(u) = 4$ and $|S(u)| = 2$ for $\deg(u) = 2$.
   Additionally let $S(u^*) = \{r\}$ when considering $X_3$.
   We apply the following rules to still uncolored vertices (initially the set $(W-A) \cup \{u^*\}$) as long as any of these is applicable:
   
   \noindent
   \textbf{Rule 1}: If for some uncolored vertex $u$ three of its neighbors have the same color $a$, we color $u$ with a color in $L(u)$ different from $a$.
   
   \noindent
   \textbf{Rule 2}: If Rule~1 does not apply, but for some uncolored vertex $u$ some $u' \in S(u)$ is already colored, we color $u$ with a color from its list different from the color of $u' $.
   
   \noindent
   \textbf{Rule 3}: If neither Rule~1 nor Rule~2 applies, and the root $r$ is uncolored, consider the set of colors appearing on $N(r)$ and a color $a$ that is repeated the most in $N(r)$.
   Let $b\in L(r)-\{a\}$.
   Then $b$ is repeated at most twice in $N(r)$ since $|N(r)|\leq 5$.
   Moreover, $b$ is repeated at most once in $N(r)$ if $|N(r)|=2$ or $3$.
   Assign color $b$ to $r$.
 \end{itemize}
 
 We claim that if none of the three rules applies, then all vertices are colored.
 Indeed, if neither Rule~1 nor Rule~2 applies and some vertex $u_1$ is uncolored, we have that $u_1\neq r$ and $S(u_1)$ is uncolored, which implies $S(u_1) \subseteq (W-A) \cup \{u^*, r \}$.
 Let $u_2$ be any vertex in $S(u_1)$.
 So, $u_1, u_2 \in W$ and thus $u_2u_1$ is a path of length $1$ in $X$ with in-endvertex $u_1$ and out-endvertex $u_2$.
 As $u_2$ is uncolored and Rule~2 does not apply we have that $S(u_2)$ is uncolored.
 Continuing this way we obtain a sequence $u_1,u_2,\ldots$ of uncolored vertices such that for each $i=1,2,\ldots$,  $u_{i+1} \in S(u_i)$ and $u_{i+1}u_i$ is a path of length $1$ in $X$ with in-endvertex $u_i$ and out-endvertex $u_{i+1}$.
 As $G$ is finite, we have $u_i = u_k$ for some $i < k$, which contradicts Lemma~\ref{lem:our-configurations}\ref{enum:acyclic}, stating that $X_j$ is acyclic.
 This shows that if none of Rule~1, Rule~2, Rule~3 applies, then all vertices in $H'$ are colored.
 So, applying Rule~1--Rule~3 as long as possible colors all the remaining vertices of $G$.
 
 \medskip 
 
 Next we shall show that the produced coloring is good, or more specifically that each monochromatic components of $H'$ is a subpath of the union of two paths from $X_j$.
 Rule~1 and Rule~3 ensure that every vertex $v \in W - A$ has at most two neighbors in the same color as $v$.
 If $u^*$ exists, then $\deg(u^*) = 3$, and hence Rule~2 ensures that $e^* = ru^*$ is colored properly.
 Moreover, for every vertex $u \in W$ let $X(u)$ be the set of paths $P$ in $X_j$ containing $u$, for which the neighbor of $u$ in $P$ has the same color as $u$.
 Then Rule~1 and Rule~2 ensure that $X(u) = \emptyset$, or $X(u)$ consists of exactly one path with in-endvertex $u$, or $X(u)$ consists of at most two paths, both with out-endvertex $u$.
 
 Recall that we colored the vertices in $A$ so that no vertex in $X_j$ has a neighbor outside of $X_j$ in the same color.
 Moreover, we colored $V(E_1) \cup B \cup \{u^*\}$ in such a way that every monochromatic component of $X_j - W$ is completely contained in a path of $X_j$.
 Finally, we colored the vertices in $W$ so that every monochromatic component of $X_j$ is the union of at most two monochromatic components of $X_j - W$.
 Together this implies that every monochromatic component is contained in the union of at most two paths in $X_j$.
 
%
%
%

 To summarize, we see that our coloring is good on $V(G)-V_1$.
 Now, each path of $X_j$ is facial, i.e., has at most $8$ edges by Lemma~\ref{lem:face-length} and each monochromatic component in $V_1$ is a path contained in the union of some two paths from $X_j$.
 This monochromatic path has length at most $14$, because it is induced and hence contains at most $7$ edges from each of the two paths.
 So, our coloring is good on $V_1$.
 Finally, since no vertex of $V_1$ has the same color as its neighbor (if exists) in $V(G)-V_1$, the vertices of each monochromatic component are completely contained in $V_1$ or in $V(G)-V_1$.
 Thus the coloring is good.
 This concludes the proof of Lemma~\ref{lem:reducible-general} saying that $X_j$, $j=1,2,3,4$, is reducible.
\end{proof}

To conclude the proof of Theorem~\ref{thm:girth-6}, we see that Lemma~\ref{lem:reducible-general} shows that $G$ has a reducible configuration, contradicting the fact that $G$ is a minimal counterexample.

\section{Proof of Theorem~\ref{thm:girth-4-LB}}\label{sec:girth-4}

For every integer $t \geq 2$ we define two planar graphs of girth $4$, denoted by $A_t$ and $B_t$, respectively.
The graph $A_t$ consists of a path $P_t$ on $t$ vertices and two special vertices $u$ and $w$, such that the vertices along $P_t$ are joined by an edge alternatingly to $u$ and $w$.
For example, $A_2$ is a path on $4$ vertices and the left of Figure~\ref{fig:girth-4-LB-graphs} shows $A_5$.
The graph $B_t$ consists of $A_t$ with special vertices $u$ and $w$, and for every neighbor $v$ of $u$ there is another copy of $A_t$, with special vertices being identified with $v$ and $w$, respectively.
See the middle of Figure~\ref{fig:girth-4-LB-graphs}.
 
Note that for every $t \geq 2$ the graph $B_t$ has girth $4$ and the two special vertices $u$ and $w$ are at distance $3$ (counted by the number of edges) in $B_t$.

\begin{figure}[htb]
 \centering
 \includegraphics{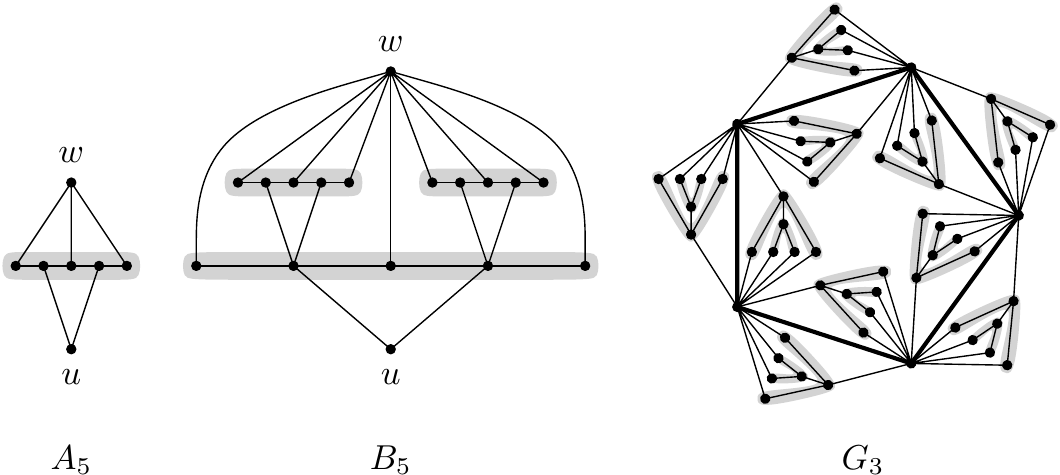}
 \caption{The graph $A_5$, $B_5$ and $G_3$.}
 \label{fig:girth-4-LB-graphs}
\end{figure}

 We construct $G_t$ inductively.
 For $t=2$, we define $G_t$ to be the $5$-cycle.
 Clearly, in any $2$-vertex coloring of $G_2$ there is a monochromatic $P_2$.
 
 For $t \geq 3$, let $G$ be a copy of $G_{t-1}$.
 We obtain $G_t$ from $G$ by considering every edge $xy$ in $G$, taking two copies $B,B'$ of $B_t$ with special vertices $u,w$ and $u',w'$, respectively, and identifying $x,u'$ and $w$, as well as $y,u$ and $w'$.
 Note that $G_t$ has girth $4$ and is indeed planar: We can embed $B$ and $B'$ on different ``sides'' of the edge $xy$, as in the right of Figure~\ref{fig:girth-4-LB-graphs}.
 
 Now, fix any $2$-vertex coloring of $G_t$ and consider the inherited coloring of $G_{t-1}$.
 By induction hypothesis there is a monochromatic copy $Q$ of $P_{t-1}$ in $G_{t-1}$, say in color $1$.
 Let $x$ be an endpoint of $Q$ and $y$ be the neighbor of $x$ in $Q$.
 Consider the copy $B$ of $H_t$ where $x$ is identified with $w$ and $y$ is identified with $u$ and the copy $A$ of $A_t$ in $B_t$ with special vertices $x=w$ and $u=y$.
 
 If the copy $P$ of $P_t$ in $A$ is not monochromatic, then at least one vertex $v$ of $P$ has color $1$.
 If $v$ is a neighbor of $x=w$, we have can extend $Q$ by $v$ into monochromatic $P_t$ in color $1$.
 Otherwise, $v$ is a neighbor of $u=y$ and we consider the copy $A'$ of $A_t$ with special vertices $u'=v$ and $w' = x$.
 Again, if the copy $P'$ of $P_t$ in $A'$ is not monochromatic, then at least one vertex $v'$ of $P'$ has color $1$.
 If $v'$ is a neighbor of $x$, then $Q \cup v'$ is a monochromatic $P_t$ in color $1$.
 Otherwise $v'$ is a neighbor of $v$, and $Q \cup \{v,v'\}$ forms a monochromatic $P_t$ in color $1$.
\qed



\section{Conclusions and open questions}\label{sec:conclusions}

In this paper, we proved that for any planar graph of girth $6$ and any assignment of lists of $2$ colors to each vertex, there is a coloring from these lists such that monochromatic components are paths of lengths at most $14$.
This extends a corresponding recent result of Borodin \textit{et al.}~\cite{BKY13} for planar graphs of girth $7$.
Our result can be interpreted as a statement about linear arboricity with short paths.

The proof uses discharging and reducible configurations.
Compared to most of the previous discharging proofs, where the reducible configurations are small, here, the reducible configuration can be arbitrarily large.
A similar approach was used by Havet and Sereni~\cite{HS06}, who argued that every graph of maximum average degree less than $3$ (which includes planar graphs of girth at least $6$) has a $2$-defective $2$-list-coloring.
The main difference between this proof and the proof of Theorem~\ref{thm:girth-6} is that Havet and Sereni can assume that in a minimal counterexample every edge is incident to a vertex of degree at least $4$.
Indeed, if there are two adjacent vertices $u,v$ of degree at most $3$ each, then a $2$-defective $2$-list-coloring of $G - \{u,v\}$ can easily be extended to a $2$-defective $2$-list-coloring of $G$.
Such a reduction does not work in our case, since we can not bound the length of a longest monochromatic path.
The reducible configurations of Havet and Sereni are not only simpler (they contain no vertices of degree $3$), with their coloring of such a configuration one can get arbitrarily long monochromatic paths.
Thus, our Lemma~\ref{lem:reducible-general} requires less and proves more then the corresponding statement of Havet and Sereni~\cite[Lemma 2]{HS06}.
 
According to Table~\ref{tab:overview} the remaining open questions concern $2$-colorings of planar graphs of girth $5$ or $6$.
Figure~\ref{fig:girth-5-example} shows a planar graph of girth $5$ that contains a monochromatic $P_3$ in every $2$-coloring, i.e., $k_d(5,2) \geq 2$, $k_f(5,2) \geq 3$, and $k_p(5,2) \geq 4$.
Indeed, we may assume without loss of generality, that in a given $2$-coloring vertices $u$ and $v$ both have color $1$.
Then $u_i$ and $v_i$, for $i = 1,2,3$, have color $2$ or there is a monochromatic $P_3$ in color $1$.
Similarly, $w_1,w_2,w_3$ have color $1$ or there is a monochromatic $P_3$ in color $2$.
But then these three vertices form a monochromatic $P_3$ in color $1$.

\begin{figure}[htb]
 \centering
 \includegraphics{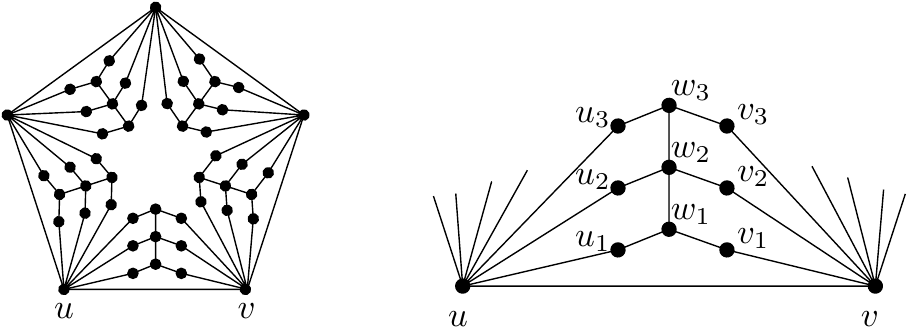}
 \caption{Every $2$-coloring contains a monochromatic path on $3$ vertices.}
 \label{fig:girth-5-example}
\end{figure}

However, to the best of our knowledge, it is open whether $k_f(5,2)$ and $k_p(5,2)$ are finite.
On the other hand, it is still possible that every planar graph of girth $5$ and $6$ admits a $2$-coloring where every monochromatic component is a subgraph of $P_3$ and $P_2$, respectively.

%

\bibliographystyle{abbrv}
\bibliography{literature}

\end{document}